\documentclass{amsart}
\usepackage{amsmath}
\usepackage{amssymb}
\usepackage{tabularx}
\usepackage{enumerate}
\usepackage{mathrsfs}
\usepackage{appendix}
\numberwithin{equation}{section}
\newtheorem{thm}{Theorem}[section]
\newtheorem{lem}[thm]{Lemma}
\newtheorem{cor}[thm]{Corollary}
\newtheorem{prop}[thm]{Proposition}
\newtheorem{defn}[thm]{Definition}
\newtheorem{exmp}[thm]{Example}

  \newenvironment{propbis}[1]
  {\renewcommand{\thethm}{\ref{#1}$'$}%
   \addtocounter{thm}{-1}%
   \begin{prop}}
  {\end{prop}}
\newcommand{\qj}{\left[\frac{q_{n+1}-j}{q_n}\right]}
\newcommand{\abs}[1]{\left\lvert #1 \right\rvert}
\newcommand{\bR}{\mathbb{R}^+}
\newcommand{\bQ}{\mathbb{Q}}
\newcommand{\bN}{\mathbb{N}}
\newcommand{\irr}{\bR\setminus\bQ}
\newcommand{\ovl}{\overline}
\newcommand{\udl}{\underline}
\newcommand{\sE}{\mathscr{E}}
\newcommand\comp{\leavevmode\raise.2ex\hbox{${\scriptstyle\mathchar"020E}$}}

\title[On the one-sided boundedness of the local discrepancy]{On the one-sided boundedness of the local discrepancy of $\{n\alpha\}$-sequences}

\begin{document}
\author{Jiangang Ying}
\address{School of Mathematics, Fudan University, Shanghai 200433, China.}
\email{jgying@fudan.edu.cn}
\author{Yushu Zheng}
\address{Shanghai Center for Mathematical Sciences, Fudan University, Shanghai 200433, China.}
\email{yszheng666@gmail.com}

\maketitle
\begin{abstract}
The main interest of this article is the one-sided boundedness of the local discrepancy of $\alpha\in\mathbb{R}\setminus\mathbb{Q}$ on the interval $(0,c)\subset(0,1)$ defined by
\[D_n(\alpha,c)=\sum_{j=1}^n 1_{\{\{j\alpha\}<c\}}-cn.\]
We focus on the special case $c\in (0,1)\cap\mathbb{Q}$. Several necessary and sufficient conditions on $\alpha$ for $(D_n(\alpha,c))$ to be one-side bounded are derived. Using these, certain topological properties are given to describe the size of the set
\[O_c=\{\alpha\in \irr: (D_n(\alpha,c)) \text{ is one-side bounded}\}.\]
\end{abstract}
\section{introduction}
For $c\in (0,1)$, an integer $n\ge 0$, and a sequence
$\omega=(\omega_n:n\ge 1)$, the local discrepancy (sequence)
 of $\omega$ at $c$ is defined to be
$$D_n(\omega,c):=\sum_{j=1}^n 1_{[0,c)}(\omega_j)-nc.$$
Given  $\theta\in [0,1)$ and
$\alpha\in \irr$, the sequence  $\eta(\alpha,\theta)=(\eta_n(\alpha,\theta))=(\{n\alpha+\theta\}:n\ge 1)$,
where $\{\cdot\}$
denotes the fractional part, is called
the sequence driven by irrational rotation with starting point $\theta$ and
rotation parameter $\alpha$.
The boundedness of the local discrepancy (sequence) of this sequence
$$D_n(\eta,c)=D_n(\eta(\alpha,\theta),c), \ n\ge 1,$$
has been a frequently disscussed topic in discrepancy theory. The following 
well-known theorem was proved by Hecke \cite{hecke1922analytische}, Ostrowski~\cite{ostrowski1922bemerkungen} (sufficiency) and  Kesten~\cite{kesten1966conjecture} (necessity).
\begin{thm}\label{kesten}
For any $\theta$, the local discrepancy $(D_n(\eta, c))$ is bounded if and only if for some integer $n$, $\{n\alpha\}=c$.
\end{thm}
Theorem \ref{kesten} gives a necessary and sufficient condition for $(D_n(\eta, c))$  to be bounded.
Following this theorem, a natural question is, while $(D_n(\alpha, c))$  is unbounded, whether it is one-side bounded, i.e., bounded above or below.
It was shown by Vera T S\'os~\cite{sosvera1976} that one can construct $\alpha$ and $c$, where $c$ is not an integral multiple (mod 1) of $\alpha$, such that $(D_n(\eta, c))$ is bounded from below (though unbounded by Theorem \ref{kesten}). Further results can be found in literature such as \cite{dupain1977intervalles}, \cite{dupain1980one},\cite{halasz1976remarks} and \cite{furstenberg1973prime}.

From now on, we will focus on the special case where $\theta=0$ and $c=h/k\in (0,1)\cap\bQ$ ($h,k\in \bN$ are coprime). Simply denote $D_n(\alpha,c):=D_n(\eta(\alpha,0),c)$. We will use standard continued fraction notations. For $\alpha\in\irr$, let $\alpha=[a_0;a_1a_2a_3\cdots a_n\cdots]$ be the continued fraction of $\alpha$, where $(a_n:n\ge0)$ are by definition the partial quotients of $\alpha$. The convergents $(p_n/q_n:n\ge -2)$ of $\alpha$ are defined by the following recursive equality:
\begin{align*}
    \left\{
    \begin{aligned}
    &p_{-2}=0, p_{-1}=1, q_{-2}=1, q_{-1}=0;\\
    &p_n = a_n p_{n-1} + p_{n-2}\text{ and }q_n = a_n q_{n-1} + q_{n-2}, \text{ for } n\ge 0.
    \end{aligned}\right.
\end{align*}
The main results in this article are stated as follows.
\begin{thm}\label{main}
    For $\alpha \in \irr$, $c=h/k$, let $(a_n), (p_n/q_n)$ be the partial quotients and convergents of $\alpha$ respectively. Then the following statements are equivalent. 
\begin{enumerate}
\item  $(D_n(\alpha,c))$ is bounded from above (resp.below), i.e. $\sup D_n(\alpha,c) <+\infty$ (resp. $\inf D_n(\alpha,c) >-\infty$);
\item there exists even (resp. odd) $m\ge -1$ such that for $n\ge 0$, $k\mid q_{m+2n}$;
\item there exists even (resp. odd)  $m\ge -1$ such that $(a_0,a_1,\cdots,a_m)$ is of type-$k$ and $k\mid a_{m+2n}$ for all $n\ge1$.
\end{enumerate}
(The definitions of tuples of type-$k$ will be introduced in \S\ref{patterntheory}.)
\end{thm}
Let $O=O_c$ be the set of $\alpha\in\irr$ with $(D_n(\alpha,c))$ one-side bounded.
\begin{thm}\label{sizeofO}
\begin{enumerate}
\item $O$ is dense in $\bR$.
\item $O$ is of the 1st category.
\item The Hausdorff dimension $\dim_H(O)\in(0,1)$. Hence $O$ is uncountable, of Lebesgue measure 0 and totally disconnected.
\end{enumerate}
\end{thm}
Let us review some other related works. In \cite{boshernitzan2009continued}, the authors study the set
\[\mathcal{H}=\mathcal{H}(c)=\left\{\alpha\in \bR:D_n(\alpha,c)\ge 0\text{ for all } n\ge0\right\}\]
for $c=1/k$. The characterizations of $\alpha\in\mathcal{H}$, via the partial quotients and convergents of $\alpha$ respectively, are obtained. Also, it is shown that $\mathcal{H}$ is of positive Hausdorff dimension. If we set 
\[\hat{O}=\hat{O}_c=\{\alpha\in O:(D_n(\alpha,c))\text{ is bounded from below}\},\]
then it is direct from the definition that for $c=1/k$, $\mathcal{H}\subset \hat{O}$. Comparing Theorem \ref{main} and \cite[Theorem 1]{boshernitzan2009continued}, the relation between elements in $\mathcal{H}$ and $\hat{O}$ can be described as follows. For $\alpha\in \hat{O}$, let
\[m_1(\alpha):=\inf\{m\ge -1:\text{$k\mid a_{m+2n}$ for all $n\ge1$}\}.\]
The necessary and sufficient condition for $\alpha$ further in $\mathcal{H}$ is $m_1(\alpha)=-1$.

In \cite{roccadas2011local}, some expressions for $\sup_{1\le n\le q_m} D_n(\alpha,c)$ and $\inf_{1\le n\le q_m} D_n(\alpha,c)$ are given respectively. From that, the sufficient condition in Theorem \ref{main} (2) for $D_n$ being bounded from above (resp. below) can be derived immediately. However, to obtain the necessary part, a more sophisticated analysis is needed. Besides the results in Theorem \ref{main}, in this paper we give a description on how the path $(D_n(\alpha,c))$ moves along with time $n$, that provides an intuition on the formulation of the one-sided boundedness.

The article is organized as follows. A recursive description of the moving mechanism of $(D_n(\alpha,c))$ is formulated in \S2. Using this, we obtain in \S3 a necessary and sufficient condition for $\alpha$ to have one-side bounded path via the convergents of $\alpha$. \S4 is devoted to a pattern theory that leads to a necessary and sufficient condition via the partial quotients of $\alpha$. As an application, in \S5 we give certain topological properties to describe the size of the set of $\alpha$ with one-side bounded paths.
\section{description of the path}\label{pathdescribe}
Fix $\alpha\in\irr$. Let $(p_n/q_n)$ and $(a_n)$ be the convergents and partial quotients of $\alpha$ respectively and $I=(0,c)$, with $c=h/k\in(0,1)$ ($h,k\in \bN$ are coprime). The goal of this section is to give an intuitive description of the path $\xi_n:=1_{\{n\alpha\in I\}}$ and $D_n=\sum_{j=1}^n \xi_j-cn$.

It is well-known that $\abs{\alpha-\frac{p_n}{q_n}}<\frac{1}{q_nq_{n+1}}$. The following lemma is an immediate consequence.
\begin{lem}\label{roughposition}
    Each interval $\left(\frac{r}{q_n},\frac{r+1}{q_n}\right)$, $r=0,1,\cdots,q_n-1$, contains exactly one point $\{j\alpha\}$ with $1\le j\le q_n$. For each $1\le j\le q_n$, there exists $0\le r< q_n$ such that
    \[\left\{\{(lq_n+j)\alpha\}:0\le l\le \left[\frac{q_{n+1}-j}{q_n}\right]\right\}\subset \left(\frac{r}{q_n},\frac{r+1}{q_n}\right).\]
\end{lem}
Let $\lambda^n_i$ ($i=0,1,\cdots, q_n-1$) be the unique integer $\in[1,q_n]$, such that 
\begin{align*}
    \lambda^n_i p_n\equiv i \pmod{q_n}.
\end{align*}
In particular, $\lambda^n_0=q_n$. It follows from the basic properties of convergents that
\begin{align*}
    \left\{
    \begin{aligned}
    &\text{when $n$ is odd, }\{\lambda^n_i \alpha\}\in\left(\frac{i-1}{q_n},\frac{i}{q_n}\right)\,\forall \,i\in[1,q_n-1]\text{ and }\{\lambda^n_0 \alpha\}\in\left(\frac{q_n-1}{q_n},1\right);\\
    &\text{when $n$ is even, }\{\lambda^n_i \alpha\}\in\left(\frac{i}{q_n},\frac{i+1}{q_n}\right)\,\forall\,i\in[0,q_n-1].
    \end{aligned}
    \right.
\end{align*}

When $k\mid q_n$, it is direct from Lemma \ref{roughposition} that
\begin{prop}
    If $k\mid q_n$, then for $1\le j\le q_n$ and $0\le l\le \qj$,
    \begin{enumerate}
        \item \begin{enumerate}[(a)]
            \item when $n$ is odd,
        \begin{align*}
            \left\{\begin{array}{ll}
            \xi_{\lambda^n_i}=1,&\text{ for }i\in [1,cq_n];\\
            \xi_{\lambda^n_i}=0,&\text{ for }i\in[cq_n+1,q_n-1]\cup\{0\};
            \end{array}\right.
        \end{align*}
            \item when $n$ is even,
        \begin{align*}
            \left\{\begin{array}{ll}
            \xi_{\lambda^n_i}=1,&\text{ for }i\in[0,cq_n-1];\\
            \xi_{\lambda^n_i}=0,&\text{ for }i\in [cq_n,q_n-1];
            \end{array}\right.
        \end{align*}
        \end{enumerate}
        \item $\xi_{lq_n+j}=\xi_j$; $D_{q_n}=0\text{ and }D_{lq_n+j}=D_j.$
    \end{enumerate}
\end{prop}

Now we turn to the case $k\nmid q_n$. Most $\xi_j$'s ($j=1,\cdots,q_{n+1}$) are clear by Lemma \ref{roughposition} except those $j$'s with $\{j\alpha\}\in\left(\frac{[cq_n]}{q_n},\frac{[cq_n]+1}{q_n}\right)$. We only give the description of the path in the case $n$ is even and it is analogous for the case $n$ is odd. Note that when $n$ is even, $\alpha> p_n/q_n$. We know that
\begin{align*}
    &j\in [1,q_{n+1}] \text{ and }\{j\alpha\}\in \left(\frac{[cq_n]}{q_n},\frac{[cq_n]+1}{q_n}\right)\\
    &\iff 
    j=\lambda_{[cq_n]}^n+lq_n,\text{ for some }0\le l\le \left[\frac{q_{n+1}-[cq_n]}{q_n}\right].
\end{align*}
Note that
\[\left(\{(\lambda_{[cq_n]}^n+lq_n)\alpha\}:0\le l\le \left[\frac{q_{n+1}-\lambda_{[cq_n]}^n}{q_n}\right]\right)\]
is an increasing arithmetic sequence. Set
\[l_n:=\inf\left\{0\le l\le \left[\frac{q_{n+1}-\lambda_{[cq_n]}^n}{q_n}\right]:\{(\lambda_{[cq_n]}^n+lq_n)\alpha\}>c\right\},\]
with the convention that $\inf\emptyset=\infty$.
\begin{prop}\label{knmidrecursive}
    If $n$ is even and $k\nmid q_n$, then for $1\le j\le q_n$ and $0\le l\le \left[\frac{q_{n+1}-j}{q_n}\right]$,
    \begin{enumerate}
        \item $\xi_{\lambda_{[cq_n]}^n}=1-1_{\{l_n=0\}}$, and
        \begin{align*}
            \xi_{\lambda^n_i}=\left\{\begin{array}{ll}
            1,&\text{ for }i\in[0,[cq_n]-1];\\
            0,&\text{ for }i\in[[cq_n]+1,q_n-1];
            \end{array}\right.
        \end{align*}
        \item $\xi_{lq_n+\lambda_{[cq_n]}^n}=1-1_{\{l\ge l_n\}}$; when $j\neq \lambda_{[cq_n]}^n$, $\xi_{lq_n+j}=\xi_j$; For the sequence $(D_n)$,
        \begin{enumerate}[(a)]
            \item when $l_n=0$, $D_{q_n}=-\left\{\frac{hq_n}{k}\right\}$ and $D_{lq_n+j}=l D_{q_n}+D_j$;
            \item when $l_n\ge1$, $D_{q_n}=1-\left\{\frac{hq_n}{k}\right\}$ and 
            \[D_{lq_n+j}=l D_{q_n}+D_j-1_{\{l\ge l_n\}}(l-l_n+1_{\{j\ge \lambda_{[cq_n]}^n\}}).\]
        \end{enumerate}
    \end{enumerate}
\end{prop}
\begin{proof}
Most of the above results are straightforward. We only give an explanation for the value of $D_{q_n}$. It is simple to see from (1) that $\abs{D_{q_n}}<1$, and $D_{q_n}<0$ if $l_n=0$ while $D_{q_n}>0$ if $l_n\ge 1$. Further note that
\[kD_{q_n}=k\sum_{j=1}^{q_n}\xi_j -hq_n\equiv -hq_n \pmod{k},\]
which leads to the conclusion.
\end{proof}
By Proposition \ref{knmidrecursive}, the path of $(D_n)$ has the following description. For $n$ even and $k\nmid q_n$, define $\hat{D}^{(n)}:=(\hat{D}_j^{(n)}:j=1,\cdots,q_n)$ (resp. $\check{D}^{(n)}:=(\check{D}_j^{(n)}:j=1,\cdots,q_n)$) as:
\[\hat{D}_j^{(n)}=\sum_{i=1}^j \hat{\xi}_i-cj\ \left(\text{resp. }\check{D}_j^{(n)}=\sum_{i=1}^j \check{\xi}_i-cj\right),\]
where for $i=1,\cdots,q_n$,
\begin{align*}
    \hat{\xi}_{\lambda^n_i}=\left\{\begin{array}{ll}
    1,&\text{ for }i\in\left[0,\lambda_{[cq_n]}^n\right];\\
    c-1,&\text{ for }i\in\left[\lambda_{[cq_n]}^n+1, q_n-1\right],
    \end{array}\right.
\end{align*}
and $\check{\xi}_{\lambda^n_i}=\hat{\xi}_{\lambda^n_i} -1_{\{i=[cq_n]\}}$. In particular,
\[\hat{D}_{q_n}^{(n)}=\check{D}_{q_n}^{(n)}+1,\  \hat{D}_{q_n}^{(n)}=1-\left\{\frac{hq_n}{k}\right\}\text{ and }\check{D}_{q_n}^{(n)}=-\left\{\frac{hq_n}{k}\right\}.\]
It follows from Proposition \ref{knmidrecursive} that for $l\in[0,l_n)$, 
\[\left(\xi_{lq_n+i}:i=1,\cdots,q_n\right)=\left(\hat{\xi}_j:i=1,\cdots,q_n\right),\]
and for $l\in[l_n,a_{n+1}]$,
\[\left(\xi_{lq_n+i}:i=1,\cdots,q_n \wedge (q_{n+1}-lq_n)\right)=\left(\check{\xi}_j:i=1,\cdots,q_n \wedge (q_{n+1}-lq_n)\right).\]
As a result, the path of $(D_n)$ can be constructed by the translation of $\hat{D}^{(n)}$ and $\check{D}^{(n)}$. Precisely, we call the path $\left(D_n:n\in\left[(l-1)q_n+1,lq_n\right]\right)$ the $l$-th period before $q_{n+1}$, if $lq_n<q_{n+1}$. Then if $1\le l\le l_n$, the $l$-th period is obtained by lifting the path, $\left(\hat{D}^{(n)}_j:1\le j\le q_n\right)$, $(l-1)D_{q_n}$ units up; if $l> l_n$, the $l$-th period is obtained by lifting the path, $\left(\check{D}^{(n)}:1\le j\le q_n\right)$, $\left(l_n\hat{D}^{(n)}_{q_n}+(l-l_n-1)\check{D}^{(n)}_{q_n}\right)$ units up. The path of $(D_n:1\le n\le q_{n+1})$ consists of $a_{n+1}$ complete periods and an incomplete period which contains only $q_{n-1}$ steps.

Now we directly give the analogous results for the case where $n$ is odd and $k\nmid q_n$. Let
\[l_n:=\inf\left\{0\le l\le \left[\frac{q_{n+1}-\lambda_{[cq_n]+1}^n}{q_n}\right]:\{(lq_n+\lambda_{[cq_n]+1}^n)\alpha\}<c\right\},\]
where $\lambda_{q_n}^n:=\lambda^n_0$, if $[cq_n]=q_n-1$.
\begin{propbis}{knmidrecursive}
If $n$ is odd and $k\nmid q_n$, then for $1\le j\le q_n$ and $0\le l\le \left[\frac{q_{n+1}-j}{q_n}\right]$,
\begin{enumerate}
        \item $\xi_{\lambda_{[cq_n]+1}^n}=1_{\{l_n=0\}}$, and
        \begin{align*}
            \xi_{\lambda^n_i}=\left\{\begin{array}{ll}
            1,&\text{ for }i\in[1,[cq_n]];\\
            0,&\text{ for }i\in[[cq_n]+2,q_n-1]\cup\{0\};
            \end{array}\right.
        \end{align*}
        \item $\xi_{lq_n+\lambda_{[cq_n]+1}^n}=1_{\{l\ge l_n\}}$; when $j\neq \lambda_{[cq_n]+1}^n$, $\xi_{lq_n+j}=\xi_j$; For the sequence $(D_n)$,
        \begin{enumerate}[(a)]
            \item when $l_n=0$, $D_{q_n}=1-\left\{\frac{hq_n}{k}\right\}$ and $D_{lq_n+j}=l D_{q_n}+D_j$;
            \item when $l_n\ge1$, $D_{q_n}=-\left\{\frac{hq_n}{k}\right\}$ and 
            \[D_{lq_n+j}=l D_{q_n}+D_j+1_{\{l\ge l_n\}}(l-l_n+1_{\{j\ge \lambda_{[cq_n]+1}^n\}}).\]
        \end{enumerate}
    \end{enumerate}
\end{propbis}

In the next corollary, we compare the path $(D_n)$ backwards from the $q_{n+1}$-th step and the $q_{n-1}$-th step (resp. the $lq_n$-th step and the $q_n$-th step) respectively, which is obtained directly from Proposition \ref{knmidrecursive} (2) and will be frequently used in \S\ref{sectionosb}.

\begin{cor}\label{backwards}
When $n$ is even and $k\nmid q_n$, 
\begin{enumerate}
    \item for $j\in[0,q_{n-1}]$, ($D_0:=0$)
\[(D_{a_{n+1}q_n+j}-D_{q_{n+1}})-(D_j-D_{q_{n-1}})=\left\{\begin{aligned}
1,&\text{ if }1\le l_n<\infty \text{ and }j< \lambda_{[cq_n]}^n\le q_{n-1};\\
0,&\text{ otherwise};
\end{aligned}\right.\]
    \item for $l\in[1,a_{n+1}]$ and $j\in[0,q_n]$, 
\[(D_{(l-1)q_n+j}-D_{lq_n})-(D_j-D_{q_n})=\left\{\begin{aligned}
1,&\text{ if }1\le l_n<\infty \text{ and }j< \lambda_{[cq_n]}^n\le q_{n-1};\\
0,&\text{ otherwise}.
\end{aligned}\right.\]
\end{enumerate}
\end{cor}
\section{one-sided boundedness}\label{sectionosb}
In this section, we will derive a necessary and sufficient condition on $\alpha$ with $(D_n(\alpha,c))$  one-side bounded. Define the level path of $(D_n)$ as
\[\ovl{D}_n:=\sup\{D_j:0\le j\le n\}\text{ and }\udl{D}_n:=\inf\{D_j:0\le j\le n\}.\]
Note that $D_n\in \frac{1}{k}\bN$. Hence $\sup D_n<\infty$ (resp. $\inf D_n>-\infty$) if and only if $\ovl{D}_n$ increases (resp. $\udl{D}_n$ decreases) only for finitely many times. First we give a useful lemma.

\begin{lem}\label{accumulation}
\begin{enumerate}
    \item If $D_{q_n}>0$ for infinitely many $n$, then $\sup D_n=\infty$;
    \item If $D_{q_n}<0$ for infinitely many $n$, then $\inf D_n=-\infty$.
\end{enumerate}
\end{lem}
\begin{proof}
    We prove (1) only. Set 
    \begin{align*}
        \lambda^n:=\left\{
        \begin{array}{ll}
        \lambda_{[cq_n]+1}^n,&\text{ when $n$ is odd};\\
        \lambda_{[cq_n]}^n,&\text{ when $n$ is even}.
        \end{array}
        \right.
    \end{align*}
    Note that if $D_{q_n}>0$, then $k\nmid q_n$ and $D_{q_n+j}>D_j$ for $j\in[1,q_{n-1}\wedge \lambda^n]$, and hence $\ovl{D}_{q_n+1}>\ovl{D}_{q_{n-1}\wedge \lambda^n}$. Now for the proof of $\sup D_n=\infty$, it suffices to show $\lambda^n\rightarrow \infty$, since $\lambda^n\rightarrow \infty$ implies $\ovl{D}_n$ increases for infinitely many times.  
    
    By definition, we readily check that when $n$ is odd (resp. even), $\lambda^n$ is the unique integer $\in [1,q_n]$, such that $\left\{\frac{\lambda^np_n-1}{q_n}\right\}<c<\left\{\frac{\lambda^np_n}{q_n}\right\}$ (resp. $\left\{\frac{\lambda^np_n}{q_n}\right\}<c<\left\{\frac{\lambda^np_n+1}{q_n}\right\}$). Fixing $n$, for $j=1,\cdots,\lambda^n$, $\left\{\frac{jp_N-1}{q_N}\right\}$, $\left\{\frac{jp_N}{q_N}\right\}$ and $\left\{\frac{jp_N+1}{q_N}\right\}$ all converge to $\{j\alpha\}$ as $N\uparrow \infty$. Hence, there exists $M>0$, for $N>M$ and $j=1,\cdots,\lambda^n$, 
    \[\frac{1}{q_N}<\left\{\frac{jp_N}{q_N}\right\}<\frac{q_N-1}{q_N}\text{ and }c\notin \left(\left\{\frac{jp_N-1}{q_N}\right\},\left\{\frac{jp_N+1}{q_N}\right\}\right).\]
    It follows that $\lambda^N>\lambda^n$, that leads to the conclusion.
\end{proof}
A sequence $(b_n)$ is called eventually a multiple of $k$ if there exists $N$, such that for $n\ge N$, $k\mid b_n$. The next theorem gives a sufficient condition for the one-sided boundedness of $(D_n)$.
\begin{thm}\label{eventuallykosb}
If $(q_{2n})$ (resp. $(q_{2n+1})$) is eventually a multiple of $k$, then $\sup D_n<\infty$ (resp. $\inf D_n>-\infty$) and $\inf D_n=-\infty$ (resp. $\sup D_n=\infty$).
\end{thm}
\begin{proof}
We only prove the case $q_{2n+1}$ is eventually a multiple of $k$. Note that consecutive $q_n$ must be coprime. So for $n$ large enough, $k\nmid q_{2n}$. First we show $\inf D_n>-\infty$ and the following lemma is needed. 
\begin{lem}\label{kmidgrowth}
    \begin{enumerate}
        \item  When $k\mid q_n$, $\ovl{D}_{q_{n+1}}=\ovl{D}_{q_n}\text{ and }\udl{D}_{q_{n+1}}=\udl{D}_{q_n}$.
        \item When $n$ is even and $k\nmid q_n$, if $D_{q_{n+1}}\ge D_{q_{n-1}}$ and $D_{a_{n+1}q_n}\ge D_{q_n}$, then $\udl{D}_{q_{n+1}}=\udl{D}_{q_n}$.
    \end{enumerate}
\end{lem}
\begin{proof}
(1) is simple. For (2), by Corollary \ref{backwards} (1), for $j\in [0,q_{n-1}]$
\begin{align}\label{for2b}
    D_{a_{n+1}q_n+j}-D_j\ge D_{q_{n+1}}-D_{q_{n-1}}\ge 0.
\end{align}
In particular, taking $j=0$, we get $D_{a_{n+1}q_n}\ge 0$. It follows from Proposition \ref{knmidrecursive} (2a) that $l_n\ge 1$. Using Proposition \ref{knmidrecursive} (2b) or the description of the path in \S\ref{pathdescribe}, $D_{lq_n}\ge D_{q_n}$ for $l=a_{n+1}$ is enough to ensure $D_{lq_n}\ge D_{q_n}$ for all $l=1,\cdots,a_{n+1}$. Hence by Corollary \ref{backwards} (2), for $l\in[1,a_{n+1}]$ and $j\in[1,q_n]$, 
\[D_{(l-1)q_n+j}-D_j\ge D_{lq_n}-D_{q_n}\ge 0.\]
Combined with \eqref{for2b}, we readily see $\udl{D}_{q_{n+1}}=\udl{D}_{q_n}$.
\end{proof}

Suppose $k\mid q_{2n-1}$ and $k\mid q_{2n+1}$. Then $\udl{D}_{q_{2n}}=\udl{D}_{q_{2n-1}}$ and $D_{q_{2n-1}}=D_{q_{2n+1}}=0$. By Lemma \ref{kmidgrowth}, in order that $\udl{D}_{q_{2n+1}}=\udl{D}_{q_{2n}}$, it is enough to show $D_{a_{2n+1}q_{2n}}=1$, since $D_{q_{2n}}<1$. Using Proposition \ref{knmidrecursive} (2), we readily see that $D_{q_{2n-1}}=D_{q_{2n+1}}$ implies $1\le l_{2n}<\infty$. By Corollary \ref{backwards} (1), $D_{a_{2n+1}q_{2n}}=1_{\{\lambda^{2n}_{[cq_{2n}]}\le q_{2n-1}\}}$. Now it reduces to prove $\lambda^{2n}_{[cq_{2n}]}\le q_{2n-1}$. We readily deduce from the definition that
$\lambda^{2n}_{[cq_{2n}]}$ is the unique integer in $A:=\{1\le j\le q_{2n}:\{j\alpha\}<c\}$ such that
\[c-\left\{j\alpha\right\}=\min_{j\in A}\left(c-\{j\alpha\}\right).\]
Since $k\mid q_{2n-1}$, it is simple to see that $\lambda^{2n-1}_{cq_{2n-1}}$ also satisfies the above condition. Thus,
$\lambda^{2n}_{[cq_{2n}]}=\lambda^{2n-1}_{cq_{2n-1}}\le q_{2n-1}$ and we see that if $k\mid q_{2n-1}$ and $k\mid q_{2n+1}$, then $\udl{D}_{q_{2n+1}}=\udl{D}_{q_{2n}}=\udl{D}_{q_{2n-1}}$. It follows that if $(q_{2n+1})$ is eventually a multiple of $k$, then $\inf D_n>-\infty$. 

Now we turn to the proof of $\sup D_n=\infty$.
\begin{lem}\label{qnan}
    For $\alpha\in\irr$, let $(p_n/q_n)$ and $(a_n)$ be the convergents and partial quotients of $\alpha$ respectively. Then the following statements are equivalent for any $N\ge -1$ and $m\ge 2$.
    \begin{enumerate}
        \item $m\mid q_{N+2n}$ for $n\ge 0$;
        \item $m\mid q_N$ and $m\mid a_{N+2n}$ for $n\ge1$.
    \end{enumerate}
\end{lem}
\begin{proof}
(1)$\Rightarrow$(2): Suppose for $n\ge 0$, $m\mid q_{N+2n}$. In particular, $m\mid q_N$. Since consecutive $q_n$ are coprime and
\begin{align}\label{duction}
    q_{N+2n+2}=a_{N+2n+2}q_{N+2n+1}+q_{N+2n},
\end{align}
$m\mid a_{N+2n}$ for all $n\ge 1$.

(2)$\Rightarrow$(1): Suppose $m\mid q_N$ and $m\mid a_{N+2n}$ for $n\ge1$. Again using \eqref{duction}, we deduce inductively that $m\mid q_{N+2n}$ for $n\ge 0$.
\end{proof}
It follows from Lemma \ref{qnan} that the condition that $q_{2n+1}$ is eventually a multiple of $k$ implies that $q_{2n}$ and $k$ are coprime for large $n$, and $a_{2n+1}$ is eventually a multiple of $k$. Note also that when $q_{2n}$ and $k$ are coprime, that $a_{2n+1}\ge k$ implies $l_{2n}\ge 1$. In fact, if $a_{2n+1}\ge k$, then
\[\alpha-\frac{p_{2n}}{q_{2n}}<\frac{1}{q_{2n}q_{2n+1}}<\frac{1}{kq_{2n}^2}.\]
It follows that
\[\lambda_{[cq_{2n}]}^{2n}\left(\alpha-\frac{p_{2n}}{q_{2n}}\right)<\frac{1}{kq_{2n}}\le c-\frac{\lambda_{[cq_{2n}]}^{2n}}{q_{2n}}=\frac{h}{k}-\frac{\lambda_{[cq_{2n}]}^{2n}}{q_{2n}},\]
which implies $l_{2n}\ge 1$. Or equivalently, $D_{q_{2n}}>0$ by Proposition \ref{knmidrecursive} (2). Then $\sup D_n=\infty$ by Lemma \ref{accumulation}.
\end{proof}
Theorem \ref{eventuallykosb} gives examples of $\alpha$'s with $\sup D_n<\infty$ or $\inf D_n>-\infty$. Now we show that these $\alpha$'s are exactly all the $\alpha$'s with this condition. 
\begin{thm}
If $(q_{2n})$ (resp. $(q_{2n+1})$) is not eventually a multiple of $k$, then $\sup D_n=\infty$ (resp. $\inf D_n=-\infty$).
\end{thm}
\begin{proof}
Suppose $(q_{2n})$ is not eventually a multiple of $k$. We shall prove $\sup D_n=\infty$. 
Note that $D_{q_n}\in\left\{-\frac{k-1}{k},-\frac{k-2}{k},\cdots,\frac{k-2}{k}, \frac{k-1}{k}\right\}$, and $k\mid q_n$ implies $D_{q_{n-1}}=D_{q_{n+1}}$. It follows easily that there exist infinitely many $n$ such that $k\nmid q_{2n}$ and $D_{q_{2n+1}}\ge D_{q_{2n-1}}$. By Corollary \ref{backwards} (1), $D_{q_{2n+1}}> D_{q_{2n-1}}$ implies $\ovl{D}_{q_{2n+1}}> \ovl{D}_{q_{2n-1}}$. So we can focus on the case there exist infinitely many $n$ such that $k\nmid q_{2n}$ and $D_{q_{2n+1}}= D_{q_{2n-1}}$. By Proposition \ref{knmidrecursive} (2a), $D_{q_{2n+1}}= D_{q_{2n-1}}$ implies $1\le l_n<\infty$ and $D_{q_{2n}}>0$. Hence, $\sup D_n=\infty$ follows from Lemma \ref{accumulation}.
\end{proof}
Now we have the following necessary and sufficient condition on $\alpha$ with $D_n$ one-side bounded.
\begin{thm}\label{charactconvergent}
For $\alpha\in\irr$, $c=h/k\in(0,1)$, let $(p_n/q_n)$ be the convergent of $\alpha$. Then
$\sup D_n<\infty$ (resp. $\inf D_n>-\infty$) if and only if $(q_{2n})$ (resp. $(q_{2n+1})$) is eventually a multiple of $k$.
\end{thm}
\section{pattern decomposition theorem}\label{patterntheory}
Theorem \ref{charactconvergent} gives a characterization of $\alpha$ with $(D_n)$  one-side bounded, which is a condition imposed on convergents. Since the convergents are determined by the partial quotients, it is more natural to impose this condition equivalently on partial quotients. 

From Lemma \ref{qnan}, we see that $\sup D_n<\infty$ (resp. $\inf D_n>-\infty$) is also equivalent to there exists even (resp. odd) $N>-1$ such that $k\mid q_N$ and $k\mid a_{N+2n}$ for $n\ge 1$. For an integer $N$, let $\pi_k(N)$ be the unique integer $\in\{0,1,\cdots,k-1\}$, such that $\pi_k (N)=N\pmod{k}$. In the following, we simply write $\pi$ for $\pi_k$, since $k$ is fixed. To answer the problem raised before, it suffices to characterize the $\pi(q_N)$ in terms of $(a_0,a_1,\cdots,a_N)$.

Given an (ordered) tuple $M=(a_0,a_1,\cdots,a_m)\ (m\ge -1)$ of non-negative integers, where $m$ is called the length of $M$. When $m=-1$, $M$ is the empty tuple $()$ (in the following, we will use the notation $\sE$ to represent the empty tuple). Define $(q_n^M:-2\le n\le m)$ recursively by: 
 \begin{align}\label{tuplerecursive}
    \left\{
    \begin{aligned}
    &q^M_{-2}=1, q^M_{-1}=0;\\
    &q^M_n = a_n q^M_{n-1} + q^M_{n-2}, \text{ for } 0\le n\le m.
    \end{aligned}\right.
\end{align}
In other words, $(q_n^M:-2\le n\le m)$ are the denominators of the convergents associated to $(a_n:0\le n\le m)$. 
What we are concerned is
\[\pi\left(q^M_{m-1},q^M_m\right):=\left(\pi(q^M_{m-1}),\pi(q^M_{m})\right).\]
Let's make a few definitions.
\begin{defn} 
For $M=(a_0,\cdots,a_m)$,
\begin{enumerate}
\item A tuple with all elements $<k$ is called a pattern, and $\pi(M):=(\pi(a_0),\cdots,\pi(a_m))$ is called the pattern of $M$.
\item The character of $M=(a_0,\cdots,a_m)$ is defined by
  $$\kappa(M)=\pi(q^M_{m-1},q^M_m).$$
  \item Operation of insertion: 
for tuples $M=(a_0,\cdots,a_m)$, $N=(b_0,\cdots,b_n)$, and $-1\le j \le m$, define
$$M\!\triangleleft_j\! N=(a_0,a_1,\cdots ,a_j,b_0,b_1,\cdots,b_n,a_{j+1},\cdots,a_m),$$
i.e. the tuple induced by inserting $N$ into $M$ after $a_j$ (or before $a_0$ if $j=-1$).
The subscript $j$ in $\triangleleft_j$ may be omitted if we do not care where the tuple is inserted.
\item If $M=M_0\!\triangleleft\! N$, $N$ is called a sub-tuple of $M$ and proper if $N\not=M$. When $N$ is a sub-tuple of $M$, we also say $M$ contains $N$.
\item A non-empty tuple $N$ is called elementary if $N$ is null, in the sense that $\kappa(M)=\kappa({M\!\triangleleft_j\! N})$ for each tuple $M$ and for each $-1\le j\le m$, and minimal in the sense that $N$ does not contain a proper null sub-tuple.
\item A tuple containing no elementary tuples is called prime.
\end{enumerate}\end{defn}
In particular, $\kappa(\sE)=(1,0)$. It is easy to see that $\kappa(M)$ depends only on the pattern $\pi(M)$. Hence as far as we are concerned, $M$ is not different from $\pi(M)$.
\begin{lem}\label{containelementary}
Let $M=(a_0,\cdots,a_m)$.
If $m\ge (k^2-1)!-1$, then $M$ contains at least one elementary tuples with length $\le (k^2-1)!-1$.
\end{lem}
\begin{proof}
Note that since consecutive $q^M_n$ are coprime, $\kappa(M)\neq(0,0)$. Put
\[P_k=\{(N_1,N_2):N_1,N_2\in\{0,1,\cdots,k-1\}\text{ and }(N_1,N_2)\neq (0,0)\}.\]
Then $\#P_k=k^2-1$. Denote $P_k=\{p_1,\cdots,p_{k^2-1}\}$. Set 
\[Q_k=\{(p_{\sigma(1)},p_{\sigma(2)},\cdots,p_{\sigma(k^2-1)}):\sigma\in S_{k^2-1}\},\]
where $S_n$ is the collection of permutations on $\{1,2,\cdots,n\}$. Then $\#Q_k=(k^2-1)!$. For $a\in\bN$, $p=(N_1,N_2)\in P_k$, define $T_a:P_k \rightarrow P_k$, $T_a((N_1,N_2))=(N_2,\pi(aN_2+N_1))$. It is easy to check $T_a$ is a bijection. Hence for $q=(p_{\sigma(1)},p_{\sigma(2)},\cdots,p_{\sigma(k^2-1)})\in Q_k$, $$T_a(q):=(T_a(p_{\sigma(1)}),T_a(p_{\sigma(2)}),\cdots,T_a(p_{\sigma(k^2-1)}))\in Q_k.$$
For $M=(a_0,\cdots,a_m)$ with $m\ge (k^2-1)!-1$, fix any $q\in Q_k$ and consider the collection
\[C_k:=\left(q,T_{a_0}(q),T_{a_1}\comp T_{a_0}(q),\cdots,T_{a_m}\comp T_{a_{m-1}}\comp\cdots\comp T_{a_0}(q)\right).\]
Since $\#C_k=m+2\ge (k^2-1)!+1$, there exist $m_1>m_2\ge 0$, such that \[T_{a_{m_1}}\comp\cdots\comp T_{a_0}(q)=T_{a_{m_2}}\comp\cdots\comp T_{a_0}(q).\]
Then it is easy to see $(a_{m_2+1},a_{m_2+2}\cdots,a_{m_1})$ is elementary.
\end{proof}
\begin{cor}\label{finitelymanyele}
\begin{enumerate}
    \item There are only finite elementary or prime patterns.
    \item For any $l\ge0$, there exists $n\ge1$, such that $(\overbrace{l,\cdots,l}^n)$ is elementary.
\end{enumerate}
\end{cor}
\begin{proof}
By Lemma \ref{containelementary}, every pattern with length $\ge (k^2-1)!$ has a proper sub-tuple that is elementary. Hence, all the elementary or prime patterns have length $\le (k^2-1)!-1$, which implies (1). For (2), we take $M=(\overbrace{l,\cdots,l}^N)$ for some $N\ge (k^2-1)!-1$. Then $M$ contains an elementary tuple. Note that all the sub-tuple of $M$ is in the form $(\overbrace{l,\cdots,l}^n)$. That concludes (2).
\end{proof}
We now present a decomposition theorem.
\begin{thm}\label{decomposition}
Let
$M=(a_0,a_1,\cdots,a_m)$.
\begin{enumerate}
\item For general $M$, there exists a prime tuple $M_0$, and finite elementary tuples $(N_i:1\le i \le n)$ 
such that 
\begin{align}\label{decompformula}
    M=M_0\!\triangleleft\! N_1\!\triangleleft\! N_2\cdots \!\triangleleft\! N_n,
\end{align}
    where
    the inserting operations are executed from left to right. This expression is called a prime decomposition of $M$.
    \item $\kappa(M)=\kappa(M_0)$. 
    \end{enumerate}
\end{thm}
In general, the decomposition \eqref{decompformula} is not unique, and neither is the prime tuple $M_0$.
The character of $M$ is determined by any one of the prime tuples $M_0$, which are hence called the prime tuples of $M$. A tuple (resp. pattern) $M$ is called of type-$k$, if $\kappa(M)=(j,0)$, for some $j=1,\cdots,k-1$. According to Corollary \ref{finitelymanyele} (1), there are only finite prime patterns of type-$k$. In particular, $\sE$ is a prime pattern of type-$k$. Intuitively, all the patterns of type-$k$ can be constructed by inserting finite elementary patterns one by one to an initial prime pattern of type-$k$.

Come back to partial quotients $(a_n)$. Combining previous results, we get
\begin{thm}\label{CFcharacter}
$\sup D_n<\infty$ (resp. $\inf D_n>-\infty$) if and only if there exists even (resp. odd)  $m\ge -1$ such that $(a_0,a_1,\cdots,a_m)$ is of type-$k$ and $k\mid a_{m+2n}$ for all $n\ge1$.
\end{thm}

\begin{exmp}
In this example, we will explicitly calculate all the elementary patterns, prime patterns and prime patterns of type-$k$ when $k=2$.
\begin{prop}\label{pattern}
When $k=2$,
\begin{enumerate}
\item A tuple $M$ is elementary  if and only if its pattern is one of the following 7 patterns (comma `,'s in pattern are omitted for simplicity): (00), (111),  (0101), (1010), (011011), (110110), (101101).
\item A tuple $M$ is prime if and only if  its pattern is one of the following 16 patterns:
$\sE$,(0),(1),(11),(10),(01),
(101),(110), (011), (010), (0110), (1101), (1011), (01101), (11011), (10110).
\end{enumerate}
\end{prop}
\begin{proof}
 We will separate the proof into several steps. For the tuple $M=(a_0,\cdots,a_m)$, we set $M^h=(a_0,\cdots,a_h)$, for $-1\le h\le m$. Our first step is to build the relation between $\kappa(M^h)$ and $\kappa(M^{h+1})$. Using \eqref{tuplerecursive}, the relation is shown by the following two tables: for $-1\le h\le m-1$, 
\begin{enumerate}
\item when $\pi(a_{h+1})=1$,
\begin{center}
\begin{tabular}{|c|c|}
\hline
$\kappa(M^h)$&$\kappa(M^{h+1})$\\ \hline
(11)&(10)\\ \hline
(10)&(01)\\ \hline
(01)&(11)\\ \hline
\end{tabular}
\end{center}
\item when $\pi(a_{h+1})=0$,
\begin{center}
\begin{tabular}{|c|c|}
\hline
$\kappa(M^h)$&$\kappa(M^{h+1})$\\ \hline
(11)&(11)\\ \hline
(10)&(01)\\ \hline
(01)&(10)\\ \hline
\end{tabular}
\end{center}
\end{enumerate}
%{\color{red}\sout{It can be seen}
Next we describe this relation algebraically. %Denote $\mathscr{P}=$, and
Digitize patterns first by defining a bijection %numerically
$\Phi: \{(11),(10),(01)\}\rightarrow \mathbb{Z}/3,$  $$ \ \Phi((11))=0,\ \Phi((10))=1,\ \Phi((01))=2.$$ Then %for any parity
%$P\in \mathscr{P}$, if $(p_{h-1},p_h)=P$,
the tables above can be written into
\begin{enumerate}
\item when $\pi(a_{h+1})=1$, $\kappa(M^{h+1})=\Phi^{-1}(\Phi(\kappa(M^h))+1)$;
\item when $\pi(a_{h+1})=0$, $\kappa(M^{h+1})=\Phi^{-1}(2\Phi(\kappa(M^h)))$.
\end{enumerate}
To combine these two cases, we define $T_1,T_0:\mathbb{Z}/3\rightarrow \mathbb{Z}/3$, by
 $$T_1(x)=x+1\text{ (mod 3)},\ T_0(x)=2x\text{ (mod 3)},\ x\in \mathbb{Z}/3$$ and
for any integer $y$, $T_y=T_{\pi(y)}$.
With these notations, we have the parity relation 
$$\kappa(M^{h+1})=\Phi^{-1}(T_{a_{h+1}}(\Phi(\kappa(M^h))))=\Phi^{-1}\comp T_{a_{h+1}}\comp\Phi(\kappa(M^h)),\ h\ge -1,$$
or $\Phi(\kappa(M^{h+1}))=T_{a_{h+1}}\comp \Phi(\kappa(M^h))$.
Hence it follows recursively that 
\begin{equation}\label{rela}
\Phi(\kappa(M))=T_{a_m}\comp T_{a_{m-1}}\comp\cdots\comp T_{a_0}\comp\Phi(\kappa(\sE)),
\end{equation}
where $\Phi(\kappa(M))=0,1$ or 2 is called the character value.
Define two mappings $T_M$ and $T_{\pi(M)}$ as $$T_{M}=T_{a_m}\comp\cdots\comp T_{a_0}=T_{\pi a_m}\comp\cdots \comp T_{\pi a_0}=T_{\pi(M)},$$ and note that $\Phi(\kappa(\sE))=1$.
\begin{lem} \label{parity-pattern}
The character $\kappa(M)$ of a tuple $M$ satisfies
$$\Phi(\kappa(M))=T_{M}\comp \Phi(\kappa(\sE))=T_M(1).$$
\end{lem}
Now the second step is to apply the relation to the case in Proposition \ref{pattern}. Note that for $N=(b_0,\cdots,b_n)$ and $-1\le k \le m$,
$$M\!\triangleleft\!_k N=(a_0,a_1,\cdots,a_k,b_0,b_1,\cdots,b_n,a_{k+1},a_{k+2},\cdots,a_m).$$
Using the relation (\ref{rela}) above,
\begin{align*}
\Phi(\kappa(M\!\triangleleft\!_k N))&= T_{a_{m}}\comp T_{a_{m-1}}\comp\cdots\comp T_{a_{k+1}}\comp T_{b_n}\comp T_{b_{n-1}}\\
&\comp\cdots\comp T_{b_0}\comp T_{a_k}\comp T_{a_{k-1}}\comp\cdots\comp T_{a_0}(\Phi(\kappa(M^0)))\\
&= T_{(a_{k+1},\cdots,a_{m})}\comp T_N\comp T_{(a_0,\cdots, a_k)}\comp \Phi(\kappa(M^0)).
\end{align*}
Hence, it is easy to see that
\begin{lem} \label{lemforthm6.9}
The tuple $N$ is null if and only if $T_N$ is an identity.
\end{lem}
The third step is to use this lemma to prove Proposition \ref{pattern}. That is only a simple algebraic calculation by Lemma \ref{lemforthm6.9}. We will do it according to the length of the tuples. Two tricks will make it easier. For $-1\le h\le m-1$,
\begin{enumerate}[(a)]
    \item if $M^h$ is elementary,  $M^{h+1}$ is neither elementary nor prime;
    \item if $M^h$ is not prime, $M^{h+1}$ is neither elementary nor prime.
\end{enumerate}

Now assume $M=(a_0,\cdots,a_m)$,
\begin{enumerate}[(i)]
\item when $m=-1$: $\sE$ is prime;
\item when $m=0$: (0),(1). Not null, all prime;
\item when $m=1$: (00), (01), (11), (10). Only (00) is elementary,
$$T_{(00)}= T_{0}\comp T_{0}(x)=2\cdot(2x)=4x\equiv x \pmod{3}.$$ The other three are prime.
\item when $m=2$: 6 patterns (010), (011), (110), (111), (100), (101) need to check. (010), (011), (110), (101) are prime.
Only (111) is elementary,\\
$T_{(111)}= T_{1}\comp T_{1}\comp T_{1}(x)=x+1+1+1=x+3\equiv x \pmod{3}.$
 \item when $m=3$: 8 patterns (1101), (1100), (1011), (1010), (0111), (0110), (0101), (0100) need to check. Among them, (1101), (1011), (0110) are prime and the following are elementary,\\
$T_{(0101)}= T_{1}\comp T_{0}\comp T_{1}\comp T_{0}(x)=2(2x+1)+1=4x+3\equiv x\pmod{3},$\\
$T_{(1010)}=T_{0}\comp T_{1}\comp T_{0}\comp T_{1}(x)=2(2(x+1)+1)=4x+6\equiv x\pmod{3}.$
\item when $m=4$: 6 patterns (11011), (11010), (10111), (10110), (01101), (01100) need to check. Among them,  (11011), (10110), (01101) are prime and none is elementary.
\item when $m=5$: 6 patterns (110111), (110110), (101101), (101100), (011010), (011011) need to check.
Among them, 3 patterns (110111), (101100), (011010) contain smaller elementary patterns, so they are neither
 elementary nor prime. The other three are elementary,\\
 $T_{(011011)}= T_{1}\comp T_{1}\comp T_{0}\comp T_{1}\comp T_{1}\comp T_{0}(x)=4x+6\equiv x\pmod{3}.$\\
$T_{(110110)}= T_{0}\comp T_{1}\comp T_{1}\comp T_{0}\comp T_{1}\comp T_{1}(x)=4x+12\equiv x\pmod{3}.$\\
$T_{(101101)}= T_{1}\comp T_{0}\comp T_{1}\comp T_{1}\comp T_{0}\comp T_{1}(x)=4x+9\equiv x\pmod{3}.$
\item when $m\ge 6$: neither elementary nor prime.
\end{enumerate}
 Here only verification for 7 elementary patterns are presented. The verification of other
 patterns is similar, 
 for example, $T_{(11)}=T_1\comp T_1(x) =x+2\not\equiv x$ $\pmod{3}$, so it is not null. 
That completes the proof.\end{proof}
By definition, a prime pattern $M$ is of type-$2$ if and only if $\kappa(M)=(10)$, i.e., $\Phi(\kappa(M))=1$. Total 16 prime patterns in Proposition \ref{pattern} may be classified according to their character  values, which are calculated by the formula in Lemma \ref{parity-pattern}:
\begin{tabular}{llll}
\hline
    $\Phi(\kappa(\sE))=1$,\\ $\Phi(\kappa(0))=2$, & $\Phi(\kappa(1))=2$,\\
    $\Phi(\kappa(11))=0$, & $\Phi(\kappa(10))=1$,& $\Phi(\kappa(01))=0$,\\
    $\Phi(\kappa(101))=2$,& $\Phi(\kappa(110))=0$,&$\Phi(\kappa(011))=1$,&$\Phi(\kappa(010))=0$,\\
    $\Phi(\kappa(0110))=2$,& $\Phi(\kappa(1101))=1$,&$\Phi(\kappa(1011))=0$,\\
$\Phi(\kappa(01101))=0$,&$\Phi(\kappa(11011))=2$,& $\Phi(\kappa(10110))=0$.\\
\hline
\end{tabular}
\begin{prop} \label{class-oe} Let $M_0$ be any prime tuple of $M$.
Then $M$ is a tuple of type-$2$
if and only if $M_0\in\{\text{$\sE$,(10),(011),(1101)}\}.$
\end{prop}
\end{exmp}
\section{Size of $O$}\label{size}
Fix $c=h/k\in(0,1)$. Recall that
\begin{align*}
O=O_c=\{\alpha\in \irr: (D_n(\alpha,c)) \text{ is one-side bounded}\}.
\end{align*}
We can readily derived from Theorem \ref{CFcharacter} that $O$ is uncountable.
To exhibit the application of Theorem \ref{CFcharacter} further, we shall use it to explore the size of $O$. We will obtain topological properties as shown in Theorem \ref{sizeofO}.

For any $m \ge 0$ and a tuple $B_m=(b_0,b_1,\cdots,b_m)$, define
\begin{align}\label{definitionI}
    I(B_m)=\{\alpha\in \irr: a_i(\alpha)=b_i,\text{ for } i=0,1,\cdots,m\},
\end{align}
which is called a fundamental interval. It is known that $I(B_m)$ is an irrational interval, i.e. the intersection of an interval and $\irr$.
\begin{lem}\label{prefix}
For any tuple $B_m$, $I(B_m)\cap O\neq \emptyset$.
\end{lem}
\begin{proof}
Fix a tuple $B_m=(b_0,b_1,\cdots,b_m)$. By Theorem \ref{CFcharacter}, it suffices to construct a tuple of type-$k$ with length $\ge m+1$ and  its first $m+1$ numbers being $b_0,b_1, \cdots, b_m$. We will prove this recursively.

The statement for $m=1$ is already proved in Corollary \ref{finitelymanyele} (2). Assume the above statement holds when $m=n$. When $m=n+1$, for the first $n$ of $b_i$'s, by inductive assumption, there exists a tuple $(d_i:1\le i\le r)$, such that $(b_0,\cdots, b_n,d_1,\cdots,d_r)$ is a tuple of type-$k$. For $b_{n+1}$, by Corollary \ref{finitelymanyele} (2), there exists $j>0$ such that $(\overbrace{b_{n+1},\cdots,b_{n+1}}^j)$ is elementary. Then

 $$(b_0,\cdots, b_n,\overbrace{b_{n+1},\cdots,b_{n+1}}^j,d_1,\cdots,d_r)$$ is a tuple of type-$k$. That concludes the proof.
\end{proof}

\begin{cor} \label{densefirst}
    $O$ is dense in $\bR$.  % and hence $O$ is also dense.
\end{cor}

Next, we will show that the Hausdorff dimension of $O$ lies in (0,1). 
Define for any $m \ge -1$ and a tuple $B_m=(b_0,b_1,\cdots,b_m)$, (recall that $B_{-1}=\sE$)
\begin{align}\label{cover}
\begin{aligned}
F(B_m):=\{y\in \irr:\ a_i(y)=b_i,\text{ for } i=0,1,\cdots,m,\\\text{and }k\mid a_{m+2n}(y)\text{ for any }n\ge 1\}.
\end{aligned}
\end{align}
It is known from Theorem \ref{CFcharacter} that
\begin{align}\label{covering}
O= \bigcup_{m\ge 0}\bigcup_{b_0,b_1,\cdots,b_m\in \mathbb{N}:\atop (b_0,\cdots,b_m)\text{ is a tuple of type-$k$}}F(b_0,b_1,\cdots b_m).   
\end{align}
As shown in \cite[Lemma 1]{good1941fractional}, all subsets $F(B_m)$ have the same Hausdorff dimension. Hence it suffices to prove $\dim_H(F)\in (0,1)$, where $$F:=F(\sE)=\{\alpha\in\irr:k\mid a_{2n+1}(\alpha)\text{ for any }n\ge0\}.$$

Note that $$F\supset \{\alpha\in\irr:a_n(\alpha)\in\{k,2k\}\text{ for any }n\ge0\},$$
where the Hausdorff dimension of the set on the right hand side is known (cf. \cite[Theorem 11]{good1941fractional}) and in particular it is positive. Hence $\dim_H(F)>0$.

Now we give an estimation for upper bound.
\begin{lem}\label{upperbound}
$\dim_H(F)<1$.
\end{lem}
\begin{proof}
$F$ is the intersection of a decreasing sequence of sets $\{E_l:l\ge0\}$, where
\begin{align*}
  E_l=\{\alpha\in \irr:k\mid a_{2n+1}(\alpha)\text{ for any }0\le n\le l\}.
\end{align*}
Note that each $E_l$ is a countable union of disjoint irrational intervals. More precisely, for $l\ge 0$ and $d_0,d_1,\cdots,d_l\in k\mathbb{N}$, set
\begin{align*}
  E(d_0,d_1,\cdots,d_l)=\{\alpha\in (0,1)\setminus \mathbb{Q}:a_{2n+1}(\alpha) =d_n\text{ for any }0\le n\le l\},
\end{align*}
and then
\begin{align}
    E_l=\bigcup_{d_0,d_1,\cdots,d_l\in k\mathbb{N}} E(d_0,d_1,\cdots,d_l).
\end{align}
Since $E_l\supset F$, $\bigcup_{d_0,d_1,\cdots,d_l\in k\mathbb{N}} E(d_0,d_1,\cdots,d_l)$ is a countable cover of $F$. We will use the following fact, which can be easily derived by \cite[(57)]{khinchin1963continued}, to estimate the $c$-content of this cover for $c\in(0,1)$. For $l\ge 1$ and $d_0,d_1,\cdots,d_l\in k\mathbb{N}$,
\begin{align}\label{ratio}
   \frac{1}{3d_l^2}<\frac{\abs{E(d_0,d_1,\cdots,d_{l})}}{\abs{E(d_0,d_1,\cdots,d_{l-1})}}<\frac{2}{d_l^2}.
\end{align}
Repeatedly using \eqref{ratio}, we have
\[\abs{E(d_0,d_1,\cdots,d_l)}<2^l\prod_{r=1,2,\cdots,l}\frac{1}{d_r^2}.\]
Following this, we get the estimation for the c-content of the above cover:
\begin{align}\label{content}
    \begin{aligned}
      &\sum_{d_0,d_1,\cdots,d_l\in k\mathbb{N}} \abs{E(d_0,d_1,\cdots,d_l)}^c<2^{cl}\sum_{d_0,d_1,\cdots,d_l\in k\mathbb{N}}\prod_{r=1,2,\cdots,l}\frac{1}{d_r^{2c}}\\
    &\quad=\left[\left(2/k^2\right)^c\sum_{j\ge1}\frac{1}{j^{2c}}\right]^l\le \left(2^{-c}\sum_{j\ge1}\frac{1}{j^{2c}}\right)^l.  
    \end{aligned}
\end{align}
Let
\[g(c)=2^{-c}\sum_{j\ge1}\frac{1}{j^{2c}},\ c\in(1/2,1].\]
It is easy to see that $g$ is continuous and decreasing in $(1/2,1]$, with $g(1)=\pi^2/12<1$ and $\lim\limits_{c \rightarrow \frac{1}{2}+}g(c)=\infty$. It follows that
there exists a $c^*\in(0,1)$ such that $g(c^*)<1$, which implies the $c^*$-content in \eqref{content}  $<\left(g(c^*)\right)^l<1$. Since 
$$\lim_{l\rightarrow \infty}\sup_{d_0,d_1,\cdots,d_l\in k\mathbb{N}}|E(d_0,d_1,\cdots,d_l)|=0,$$
the $c^*$-Hausdorff measure $\mathcal{H}^{c^*}(F)\le1$ and $\dim_H(F)\le c^*<1$.
\end{proof}

%Now we turn to the lower bound.
%\begin{lem}\label{lowerbound}
%$\dim_H(O)\ge 1/2.$
%\end{lem}
%\begin{proof}
%For the lower bound, it is enough to pick a special subset of $O$ and prove the Hausdorff dimension of this subset $\ge1/2$. Set
%$$V=\{y\in\irr: a_{k}(y)=2\text{ or }4,\text{ for any }k\ge0\}.$$
%We will show $\dim_H(V)\ge1/2$.
%Let
%$$V_l=\{y\in\irr:a_{k}(y)=2\text{ or }4, \text{ for } 0\le k\le l\}.$$
%Then each $V_l$ is a finite union of disjoint irrational intervals (called $l$-th level intervals), with each interval of $V_l$ containing exactly two intervals of $V_{l+1}$ and the maximum length of $l$-th level intervals tending to 0 as $l\uparrow \infty$. We have
%$$V=\bigcap_{l\ge 0} V_l.$$
%As shown in Example 4.6 in \cite{falconer2004fractal},
%\begin{align}\label{lower}
%   \dim_H(V)\ge \liminf_{l\rightarrow \infty}\frac{(l-1)\log(2)}{-\log(2\epsilon_l)},
%\end{align}
%where $\epsilon_l$ is the infimum of the distance of two intervals of a $(l+1)$-level set which are contained in the same $l$-level set. In fact, $\epsilon_l$ is the diameter of the irrational interval $U_l$, where
%$$U_l=\{y\in\irr: a_k(y)=4,\text{ for }1\le k\le l,\text{ and } a_{k+1}=3\}.$$
%Again using (57) in \cite{khinchin1963continued}, we have $\epsilon_l=|U_l|\ge \frac{1}{9\cdot 2^{2l}}.$ Putting it in \eqref{lower}, we have $\dim_{H}(V)\ge 1/2.$
%\end{proof}
%Combining Lemma \ref{upperbound} $\&$ \ref{lowerbound}, it follows that
Now we have $\dim_H(F)\in (0,1)$. It follows that
\begin{thm}\label{dimension}
$\dim_H(O)\in(0,1)$. Hence, $O$ is uncountable, of Lebesgue measure 0 and totally disconnected.
\end{thm}
At last, we give a description of the size of $O$ in the view of category.
\begin{thm}\label{1st}
$O$ is of 1st category.
\end{thm}
\begin{proof}We again use the same notation $F(B_m)$ defined in \eqref{cover}. By \eqref{covering}, $O$ is covered by a countable union of $(F(B_m))$. It suffices to show that every $F(B_m)$ is nowhere dense. We have the following characterization of the closure of $F(B_m)$.\phantom{\qedhere}
\begin{lem}\label{closure}
For $B_m=(b_0,b_1,\cdots, b_m)$, 
$$\overline{F(B_m)}=F(B_m)\cup Q(B_m),$$
where
\begin{align*}
    Q(B_m)=\{q\in\bQ: q=[b_0;b_1\cdots b_m d_{m+1} d_{m+2}\cdots d_{m+l}]\text{ for some }l\ge1 \text{ and}\\\text{ positive integers } d_{m+1},\cdots, d_{m+l},\text{where } k\mid d_{m+2n}\text{ for }1\le n\le [l/2]\}.
\end{align*}
\end{lem}
\begin{proof}
First we shall prove $\overline{F(B_m)}\supset Q(B_m)$. In fact, for any fixed $q=[b_0;b_1\cdots b_n \\d_{m+1} \cdots d_{m+l}]\in Q(B_m)$, we can take for $r\ge 1$, \[q_r=(b_0,\cdots,b_m,d_{m+1},\cdots, d_{m+l},rk,rk,\cdots,rk,\cdots).\]
It is easy to see $q_r\in F(B_m)$ and $q_r \rightarrow q$ as $r\uparrow \infty$ and hence $q\in\overline{F(B_m)}$.

For another direction, it suffices to prove that for any $\alpha\notin F(B_m)\cup Q(B_m)$, there exists an interval $I$ containing $\alpha$ such that $I\cap F(B_m)=\emptyset$. Indeed, for such $\alpha$, we can always find a fundamental interval associated to $\alpha$ satisfying the condition. This concludes the proof.
\end{proof}
\begin{proof}[Proof of Theorem \ref{1st} continued]
By Lemma \ref{closure}, $\overline{F(B_m)}=F(B_m)\cup Q(B_m)$. Recall in the proof of Lemma \ref{upperbound}, we have shown that $\dim_H(F(B_m))<1$. Since $Q(B_m)\subseteq \mathbb{Q}$, $\dim_H(Q(B_m))=0$. It follows that $\dim_H(\overline{F(B_m)})<1$. Hence $\overline{F(B_m)}$ is a totally disconnected closed set, which is then nowhere dense.
\end{proof}
\end{proof}

\bibliographystyle{abbrv}
\bibliography{mybib}
\end{document}